\documentclass[12pt]{amsart}
\usepackage{amssymb}
\usepackage{amsmath}
\usepackage{mathrsfs}
\usepackage{amsthm}
\usepackage{enumerate}
\newtheorem{theorem}{Theorem}
\theoremstyle{definition}
\newtheorem{remark}{Remark}
\newtheorem{example}{Example}
\begin{document}
\allowdisplaybreaks
\title{Estimation of expected value of function of i.i.d.~Bernoulli random variables}
\author{March T.~Boedihardjo}
\address{Department of Mathematics, University of California, Los Angeles, CA 90095-1555}
\email{march@math.ucla.edu}
\begin{abstract}
We estimate the expected value of certain function $f:\{-1,1\}^{n}\to\mathbb{R}$. For example, with computer assistance, we show that if $\Delta$ is the Laplacian of the Cayley graph of $(\mathbb{Z}/15\mathbb{Z})\times(\mathbb{Z}/15\mathbb{Z})$ and $D$ is a diagonal $225\times 225$ matrix with entries chosen independently and uniformly from $\{-1,1\}$, then the expected value of the normalized trace of $(2I+D-\Delta)^{-1}$ is between $0.2006$ and $0.2030$.
\end{abstract}
\maketitle
Let $n\in\mathbb{N}$ and let $f:\{-1,1\}^{n}\to\mathbb{R}$. The expected value of $f$ is
\[\mathbb{E}f=\frac{1}{2^{n}}\sum_{\epsilon_{1},\ldots,\epsilon_{n}\in\{-1,1\}}f(\epsilon_{1},\ldots,\epsilon_{n}).\]
The computational complexity for calculating $\mathbb{E}f$ is $2^{n}$ times the cost of computing each $f(\epsilon_{1},\ldots,\epsilon_{n})$ plus the cost of averaging. When $f$ is regular enough, one can take samples $(X_{1}^{(i)},\ldots,X_{n}^{(i)})$, for $i=1,\ldots,p$, where $p$ is large enough ($p=1$ suffice when $f$ has concentration property \cite{McDiarmid},\cite{Talagrand}), and then with high probability, $\mathbb{E}f$ is close to $\frac{1}{p}\sum_{i=1}^{p}f(X_{1}^{(i)},\ldots,X_{n}^{(i)})$. In this case, the computational complexity reduces to $p$ times the cost of computing a value of $f(\epsilon_{1},\ldots,\epsilon_{n})$ plus the cost of averaging but with the trade off of having some risk.

In this paper, we examine the question of how much computational cost is needed in order to assert an estimation of $\mathbb{E}f$ up to an error of $\delta>0$. Here, ``assert" means establish, say, a theorem that $|\mathbb{E}f-1.5|\leq 0.1$ rather than making a statement that ``we are 90\% confident that $|\mathbb{E}f-1.5|\leq 0.1$."

We show that if $f$ has nonnegative Fourier coefficients and
\begin{equation}\label{eq1}
|f(\epsilon_{1},\ldots,\epsilon_{n})-f(\epsilon_{1},\ldots,\epsilon_{r-1},-\epsilon_{r},\epsilon_{r+1},\ldots,\epsilon_{n})|\leq\frac{c}{n},
\end{equation}
for all $\epsilon_{1},\ldots,\epsilon_{n}\in\{-1,1\}$ and $r=1,\ldots,n$, then with at least $90\%$ chance, one is able to assert an estimation of $\mathbb{E}f$, up to an error of at most $\frac{10c}{2p}$, with computational complexity being $(n+1)p^{2}$ times the cost of computing a value of $f(\epsilon_{1},\ldots,\epsilon_{n})$ plus the cost of some summations. (For the proof, see Remark \ref{rem1} below.)

The Laplacian $\Delta$ of a finite graph $G=(V,E)$ is the linear transformation on $\mathbb{C}^{V}$ defined by
\[(\Delta f)(v)=\sum_{w}(f(w)-f(v)),\]
for $f:V\to\mathbb{C}$, where the summation is over all adjacent vertices $w$ of $v$.

Suppose that $G$ is a given graph with $n$ vertices and $\lambda>0$ are given. Let $D$ be an $n\times n$ diagonal matrix with diagonal entries chosen independently and uniformly from $\{-1,1\}$. In general, it is a nontrivial task to estimate the expected eigenvalue distribution of the random Schr\"odinger operator $\lambda D-\Delta$ even when $G=(\mathbb{Z}/n\mathbb{Z})^{2}$ (see, e.g., \cite{Aizenman}).

We show that for given $\delta,\gamma>0$, with at least $90\%$ chance, one is able to assert an estimation of
\[\frac{1}{n}\mathbb{E}\circ\mathrm{Tr}[((\lambda+\gamma)I+\lambda D-\Delta)^{-1}],\]
up to an error of $\delta>0$, with computational complexity $\frac{Cn^{4}\lambda^{2}}{\delta^{2}\gamma^{4}}$, where $C>0$ is a universal constant and $\mathrm{Tr}$ is the trace. In particular, with computer assistance, we show that if $\Delta$ is the Laplacian of the Cayley graph of $(\mathbb{Z}/15\mathbb{Z})\times(\mathbb{Z}/15\mathbb{Z})$ and $D$ is a diagonal $225\times 225$ matrix with entries chosen independently and uniformly from $\{-1,1\}$, then
\[0.2006\leq\frac{1}{225}\mathbb{E}\circ\mathrm{Tr}[(2I+D-\Delta)^{-1}]\leq 0.2030.\]
(See Example \ref{example1}.)

Throughout this paper, if $n\in\mathbb{N}$ then $[n]=\{1,\ldots,n\}$; and if $A$ is an $n\times n$ matrix then $\mathrm{tr}\,A$ is $\frac{1}{n}$ times the trace of $A$. The following is the main result of this paper.
\begin{theorem}\label{thm}
Suppose that $f:\{-1,1\}^{n}\to\mathbb{R}$ and there exist $a_{S}\in[0,\infty)$, for $S\subset [n]$, such that
\[f(\epsilon_{1},\ldots,\epsilon_{n})=\sum_{S\subset[n]}a_{S}\prod_{k\in S}\epsilon_{k},\]
for all $\epsilon_{1},\ldots,\epsilon_{n}\in\{-1,1\}$. Define $g:\{-1,1\}^{n}\to\mathbb{R}$ by
\[g(\epsilon_{1},\ldots,\epsilon_{n})=\frac{1}{2}\sum_{r=1}^{n}(f(\epsilon_{1},\ldots,\epsilon_{n})-
f(\epsilon_{1},\ldots,\epsilon_{r-1},-\epsilon_{r},\epsilon_{r+1},\ldots,\epsilon_{n})).\]
Let $p\in\mathbb{N}$. For $i=1,\ldots,p$, let $X_{1}^{(i)},\ldots,X_{n}^{(i)}\in\{-1,1\}$. Then
\[0\leq\frac{1}{p^{2}}\sum_{i,j\in[p]}f(X_{1}^{(i)}X_{1}^{(j)},\ldots,X_{n}^{(i)}X_{n}^{(j)})-\mathbb{E}f\leq
\frac{1}{p^{2}}\sum_{i,j\in[p]}g(X_{1}^{(i)}X_{1}^{(j)},\ldots,X_{n}^{(i)}X_{n}^{(j)})\]
Moreover, if $X_{k}^{(i)}$, for $i=1,\ldots,p$ and $k=1,\ldots,n$, are chosen independently and uniformly from $\{-1,1\}$, then
\[\mathbb{E}\frac{1}{p^{2}}\sum_{i,j\in[p]}g(X_{1}^{(i)}X_{1}^{(j)},\ldots,X_{n}^{(i)}X_{n}^{(j)})=\frac{1}{p}g(1,\ldots,1).\]
\end{theorem}
\begin{remark}\label{rem1}
In the context of Theorem \ref{thm}, if (\ref{eq1}) above is satisfied, we have $g(1,\ldots,1)\leq\frac{c}{2}$ and by Markov's inequality, if $X_{k}^{(i)}$, for $i=1,\ldots,p$ and $k=1,\ldots,n$, are chosen independently and uniformly from $\{-1,1\}$, then with at least $90\%$ chance,
\begin{equation}\label{eq2}
\frac{1}{p^{2}}\sum_{i,j\in[p]}g(X_{1}^{(i)}X_{1}^{(j)},\ldots,X_{n}^{(i)}X_{n}^{(j)})\leq\frac{10c}{2p}.
\end{equation}
In this case, one obtains the following estimation of $\mathbb{E}f$.
\[0\leq\frac{1}{p^{2}}\sum_{i,j\in[p]}f(X_{1}^{(i)}X_{1}^{(j)},\ldots,X_{n}^{(i)}X_{n}^{(j)})-\mathbb{E}f\leq\frac{10c}{2p}.\]
In this estimation, the main cost comes from verifying that (\ref{eq2}) is satisfied, which could happen to fail. The cost of computing a value of $g(\epsilon_{1},\ldots,\epsilon_{n})$ is $(n+1)$ times the cost of computing a value of $f(\epsilon_{1},\ldots,\epsilon_{n})$ plus some cost of summations. Thus, the cost to verify (\ref{eq2}) is $(n+1)p^{2}$ times the cost of computing a value of $f(\epsilon_{1},\ldots,\epsilon_{n})$ plus some cost of summations.

Since one has already calculated $f(X_{1}^{(i)}X_{1}^{(j)},\ldots,X_{n}^{(i)}X_{n}^{(j)})$ before obtaining the value of $g(X_{1}^{(i)}X_{1}^{(j)},\ldots,X_{n}^{(i)}X_{n}^{(j)})$, the total cost of estimating $\mathbb{E}f$ is again $(n+1)p^{2}$ times the cost of computing a value of $f(\epsilon_{1},\ldots,\epsilon_{n})$ plus some cost of summations, if (\ref{eq2}) happens to be satisfied.
\end{remark}
\begin{proof}[Proof of Theorem \ref{thm}]
We have
\begin{eqnarray*}
\frac{1}{p^{2}}\sum_{i,j\in[p]}f(X_{1}^{(i)}X_{1}^{(j)},\ldots,X_{n}^{(i)}X_{n}^{(j)})&=&
\frac{1}{p^{2}}\sum_{i,j\in[p]}\sum_{S\subset[n]}a_{S}\prod_{k\in S}(X_{k}^{(i)}X_{k}^{(j)})\\&=&
\sum_{S\subset[n]}a_{S}\left(\frac{1}{p}\sum_{i\in[p]}\prod_{k\in S}X_{k}^{(i)}\right)^{2}.
\end{eqnarray*}
Since $a_{S}\geq 0$ for all $S\subset [n]$, it follows that
\[\mathbb{E}f=a_{\emptyset}\leq\frac{1}{p^{2}}\sum_{i,j\in[p]}f(X_{1}^{(i)}X_{1}^{(j)},\ldots,X_{n}^{(i)}X_{n}^{(j)}),\]
and
\begin{align}\label{eq3}
&\frac{1}{p^{2}}\sum_{i,j\in[p]}f(X_{1}^{(i)}X_{1}^{(j)},\ldots,X_{n}^{(i)}X_{n}^{(j)})-a_{\emptyset}\\=&
\sum_{\substack{S\subset[n]\\S\neq\emptyset}}a_{S}\left(\frac{1}{p}\sum_{i\in[p]}\prod_{k\in S}X_{k}^{(i)}\right)^{2}\nonumber\\\leq&
\sum_{S\subset[n]}|S|a_{S}\left(\frac{1}{p}\sum_{i\in[p]}\prod_{k\in S}X_{k}^{(i)}\right)^{2}\nonumber\\=&
\frac{1}{p^{2}}\sum_{i,j\in[p]}\sum_{S\subset[n]}|S|a_{S}\prod_{k\in S}(X_{k}^{(i)}X_{k}^{(j)}),\nonumber
\end{align}
where $|S|$ is the size of the set $S$. For $r=1,\ldots,n$,
\[f(\epsilon_{1},\ldots,\epsilon_{n})-f(\epsilon_{1},\ldots,\epsilon_{r-1},-\epsilon_{r},
\epsilon_{r+1},\ldots,\epsilon_{n})=2\sum_{\substack{S\subset[n]\\r\in S}}a_{S}\prod_{k\in S}\epsilon_{k},\]
so
\[g(\epsilon_{1},\ldots,\epsilon_{n})=\sum_{S\subset[n]}|S|a_{S}\prod_{k\in S}\epsilon_{k}.\]
Therefore, by (\ref{eq3}),
\[\frac{1}{p^{2}}\sum_{i,j\in[p]}f(X_{1}^{(i)}X_{1}^{(j)},\ldots,X_{n}^{(i)}X_{n}^{(j)})-\mathbb{E}f\leq
\frac{1}{p^{2}}\sum_{i,j\in[p]}g(X_{1}^{(i)}X_{1}^{(j)},\ldots,X_{n}^{(i)}X_{n}^{(j)}).\]

Note that $\mathbb{E}g=0$. Thus, if $X_{1}^{(i)},\ldots,X_{n}^{(i)}$, for $i=1,\ldots,p$, are chosen independently and uniformly from $\{-1,1\}$, then for every $i\in [p]$ and every $j\in [p]$ such that $i\neq j$, the random variables $X_{1}^{(i)}X_{1}^{(j)},\ldots,X_{n}^{(i)}X_{n}^{(j)}$ are independent and uniformly distributed on $\{-1,1\}$, so
$\mathbb{E}g(X_{1}^{(i)}X_{1}^{(j)},\ldots,X_{n}^{(i)}X_{n}^{(j)})=0$. Thus,
\[\mathbb{E}\frac{1}{p^{2}}\sum_{i,j\in[p]}g(X_{1}^{(i)}X_{1}^{(j)},\ldots,X_{n}^{(i)}X_{n}^{(j)})=\frac{1}{p}g(1,\ldots,1).\]
\end{proof}
\begin{example}\label{example1}
Let $\gamma,\lambda>0$. Suppose that $A$ is the adjacency matrix of a graph $G$ with $n$ vertices and $D_{G}$ is the diagonal matrix with diagonal entries being the degrees of the vertices of $G$. Then $\Delta=A-D_{G}$. For every $\epsilon=(\epsilon_{1},\ldots,\epsilon_{n})\in\{-1,1\}^{n}$, let $D_{\epsilon}$ be the $n\times n$ diagonal matrix with diagonal entries $\epsilon_{1},\ldots,\epsilon_{n}$. Define $f:\{-1,1\}^{n}\to\mathbb{R}$ by \[f(\epsilon_{1},\ldots,\epsilon_{n})=\mathrm{tr}[((\lambda+\gamma)I-\lambda D_{\epsilon}-\Delta)^{-1}],\]
where $\mathrm{tr}$ is $\frac{1}{n}$ times the trace. Note that since $(\epsilon_{1},\ldots,\epsilon_{n})$ has symmetric distribution,
\[\mathbb{E}f=\mathbb{E}\circ\mathrm{tr}[((\lambda+\gamma)I+\lambda D_{\epsilon}-\Delta)^{-1}].\]
Since $-\Delta$ is positive semidefinite, $\lambda I+\lambda D_{\epsilon}-\Delta$ is also positive semidefinite so \[|f(\epsilon_{1},\ldots,\epsilon_{n})-f(\epsilon_{1},\ldots,\epsilon_{r-1},-\epsilon_{r},\epsilon_{r+1},\ldots,\epsilon_{n})|\leq
\frac{2\lambda}{n\gamma^{2}},\]
for all $\epsilon_{1},\ldots,\epsilon_{n}\in\{-1,1\}$ and $r\in [n]$. Thus, if we define $g:\{-1,1\}^{n}\to\mathbb{R}$ by
\[g(\epsilon_{1},\ldots,\epsilon_{n})=\frac{1}{2}\sum_{r=1}^{n}(f(\epsilon_{1},\ldots,\epsilon_{n})-
f(\epsilon_{1},\ldots,\epsilon_{r-1},-\epsilon_{r},\epsilon_{r+1},\ldots,\epsilon_{n})),\]
then
\begin{equation}\label{eq4}
g(\epsilon_{1},\ldots,\epsilon_{n})\leq\frac{\lambda}{\gamma^{2}},
\end{equation}
for all $\epsilon_{1},\ldots,\epsilon_{n}\in\{-1,1\}$.

We have
\begin{align*}
&((\lambda+\gamma)I-\lambda D_{\epsilon}-\Delta)^{-1}\\=&((\lambda+\gamma)I+D_{G}-\lambda D_{\epsilon}-A)^{-1}\\=&
((\lambda+\gamma)I+D_{G})^{-1}[I-(\lambda D_{\epsilon}+A)((\lambda+\gamma)I+D_{G})^{-1}]^{-1}\\=&
\sum_{m=0}^{\infty}((\lambda+\gamma)I+D_{G})^{-1}[(\lambda D_{\epsilon}+A)((\lambda+\gamma)I+D_{G})^{-1}]^{m}.
\end{align*}
Since $D_{G}$ is diagonal with nonnegative entries, the entries of $((\lambda+\gamma)I+D_{G})^{-1}$ are nonnegative. The entries of $A$ are also nonnegative. Therefore, if we express $f(\epsilon_{1},\ldots,\epsilon_{n})$ as $\sum_{S\subset [n]}a_{S}\prod_{k\in S}\epsilon_{k}$, then $a_{S}\geq 0$ for all $S\subset [n]$.

By Theorem \ref{thm}, if $p\in\mathbb{N}$ and $X_{1}^{(i)},\ldots,X_{n}^{(i)}\in\{-1,1\}$, for $i=1,\ldots,p$, then
\[0\leq\frac{1}{p^{2}}\sum_{i,j\in[p]}f(X_{1}^{(i)}X_{1}^{(j)},\ldots,X_{n}^{(i)}X_{n}^{(j)})-\mathbb{E}f\leq
\frac{1}{p^{2}}\sum_{i,j\in[p]}g(X_{1}^{(i)}X_{1}^{(j)},\ldots,X_{n}^{(i)}X_{n}^{(j)})\]
and if the $X_{1}^{(i)},\ldots,X_{n}^{(i)}$ are chosen independently and uniformly from $\{-1,1\}$, then
\[\mathbb{E}\frac{1}{p^{2}}\sum_{i,j\in[p]}g(X_{1}^{(i)}X_{1}^{(j)},\ldots,X_{n}^{(i)}X_{n}^{(j)})\leq\frac{\lambda}{p\gamma^{2}},\]
by (\ref{eq4}). Since the cost of computing a value of $f(\epsilon_{1},\ldots,\epsilon_{n})$ is $Cn^{3}$, where the $C>0$ is universal constant, it follows that with at least $90\%$ chance, one is able to estimate $\mathbb{E}f$, up to an error of $\frac{10\lambda}{p\gamma^{2}}$, with computational complexity being $(n+1)p^{2}Cn^{3}$ plus the cost of some summations. Thus, for a given $\delta>0$, if we take $p$ to be the smallest integer greater than $\frac{10\lambda}{\gamma^{2}\delta}$, then we have that with at least $90\%$ chance, one is able to estimate $\mathbb{E}f$, up to an error of $\delta$, with computational complexity at most $\frac{C_{1}n^{4}\lambda^{2}}{\gamma^{4}\delta^{2}}$, where $C_{1}>0$ is a universal constant.

In the case when $G$ is the Cayley graph of $(\mathbb{Z}/15\mathbb{Z})\times(\mathbb{Z}/15\mathbb{Z})$ and $\gamma=\lambda=1$, i.e.,
\[f(\epsilon_{1},\ldots,\epsilon_{n})=\mathrm{tr}[(2I-D_{\epsilon}-\Delta))^{-1}],\]
if we take $p=30$, then with computer assistance, we obtain $0.2006\leq\mathbb{E}f\leq0.2030$.
\end{example}
It is easy to obtain the following result by adapting the proof of Theorem \ref{thm}.
\begin{theorem}\label{thm2}
For $S\subset [n]$, let $a_{S}\in\mathbb{C}$ and $b_{S}\in[0,\infty)$ be such that $|a_{S}|\leq b_{S}$. Define functions $f_{1},f_{2},g$ from $\{-1,1\}^{n}$ to $\mathbb{C}$ by
\[f_{1}(\epsilon_{1},\ldots,\epsilon_{n})=\sum_{S\subset[n]}a_{S}\prod_{k\in S}\epsilon_{k},\quad f_{2}(\epsilon_{1},\ldots,\epsilon_{n})=\sum_{S\subset[n]}b_{S}\prod_{k\in S}\epsilon_{k},\]
\[g(\epsilon_{1},\ldots,\epsilon_{n})=\frac{1}{2}\sum_{r=1}^{n}(f_{2}(\epsilon_{1},\ldots,\epsilon_{n})-
f_{2}(\epsilon_{1},\ldots,\epsilon_{r-1},-\epsilon_{r},\epsilon_{r+1},\ldots,\epsilon_{n})),\]
for $\epsilon_{1},\ldots,\epsilon_{n}\in\{-1,1\}$. Let $p\in\mathbb{N}$. For $i=1,\ldots,p$, let $X_{1}^{(i)},\ldots,X_{n}^{(i)}\in\{-1,1\}$. Then
\[\left|\frac{1}{p^{2}}\sum_{i,j\in[p]}f_{1}(X_{1}^{(i)}X_{1}^{(j)},\ldots,X_{n}^{(i)}X_{n}^{(j)})-\mathbb{E}f_{1}\right|\leq
\frac{1}{p^{2}}\sum_{i,j\in[p]}g(X_{1}^{(i)}X_{1}^{(j)},\ldots,X_{n}^{(i)}X_{n}^{(j)})\]
Moreover, if $X_{k}^{(i)}$, for $i=1,\ldots,p$ and $k=1,\ldots,n$, are chosen independently and uniformly from $\{-1,1\}$, then
\[\mathbb{E}\frac{1}{p^{2}}\sum_{i,j\in[p]}g(X_{1}^{(i)}X_{1}^{(j)},\ldots,X_{n}^{(i)}X_{n}^{(j)})=\frac{1}{p}g(1,\ldots,1).\]
\end{theorem}
\begin{example}
Same setting as Example \ref{example1}: $\gamma,\lambda>0$ and $A$ and $D$ are the adjacency and degree matrices of a graph $G$ with $n$ vertices, respectively. For every $\epsilon=(\epsilon_{1},\ldots,\epsilon_{n})\in\{-1,1\}^{n}$, let $D_{\epsilon}$ be the $n\times n$ diagonal matrix with diagonal entries $\epsilon_{1},\ldots,\epsilon_{n}$.

Let $d$ be the maximum degree of $G$. Let $h$ be an analytic function on a neighborhood of $\{z\in\mathbb{C}:|z-d|\leq d+\lambda+\gamma\}$. Define $f_{1}:\{-1,1\}^{n}\to\mathbb{C}$ and $f_{2}:\{-1,1\}^{n}\to\mathbb{C}$ by
\[f_{1}(\epsilon_{1},\ldots,\epsilon_{n})=\mathrm{tr}\,h(-\lambda D_{\epsilon}-\Delta),\]
and
\begin{align*}
&f_{2}(\epsilon_{1},\ldots,\epsilon_{n})\\=&\frac{d+\lambda+\gamma}{2\pi}\left(\int_{0}^{2\pi}|h(-d+(d+\lambda+\gamma)e^{it})|\,dt\right)
\mathrm{tr}[((\lambda+\gamma)I-\lambda D_{\epsilon}-\Delta)^{-1}].
\end{align*}
We want to estimate $\mathbb{E}f_{1}=\mathbb{E}\circ\mathrm{tr}\,h(\lambda D_{\epsilon}-\Delta)$.

We have
\begin{align*}
&f_{1}(\epsilon_{1},\ldots,\epsilon_{n})\\=&\mathrm{tr}\,h(-\lambda D_{\epsilon}+D_{G}-A)\\=&
\frac{1}{2\pi i}\oint_{C}h(z)\mathrm{tr}[(zI-(-\lambda D_{\epsilon}+D_{G}-A))^{-1}]\,dz\\=&
\frac{1}{2\pi i}\oint_{C}h(z)\mathrm{tr}[(zI-D_{G})^{-1}(I+(\lambda D_{\epsilon}+A)(zI-D_{G})^{-1})^{-1}]\,dz\\=&
\sum_{m=0}^{\infty}\frac{1}{2\pi i}\oint_{C}h(z)\mathrm{tr}[(zI-D_{G})^{-1}((\lambda D_{\epsilon}+A)(D_{G}-zI)^{-1})^{m}]\,dz,
\end{align*}
where $C=\{z\in\mathbb{C}:|z-d|=d+\lambda+\gamma\}$. Similarly,
\begin{align*}
&\mathrm{tr}[((\lambda+\gamma)I-\lambda D_{\epsilon}-\Delta)^{-1}]\\=&
\mathrm{tr}[(D_{G}+(\lambda+\gamma)I-\lambda D_{\epsilon}-A)^{-1}]\\=&
\mathrm{tr}[(D_{G}+(\lambda+\gamma)I)^{-1}(I-(\lambda D_{\epsilon}+A)(D_{G}+(\lambda+\gamma)I)^{-1})^{-1}]\\=&
\sum_{m=0}^{\infty}\mathrm{tr}[(D_{G}+(\lambda+\gamma)I)^{-1}((\lambda D_{\epsilon}+A)(D_{G}+(\lambda+\gamma)I)^{-1})^{m}].
\end{align*}
If $z$ is a point on the contour $C$ and $d_{v}$ is the degree of a vertex $v$, then
\[\frac{1}{|z-d_{v}|}\leq\frac{1}{|z-d|-(d-d_{v})}=\frac{1}{d_{v}+\lambda+\gamma}.\]
This means that the absolute values of the entries of $(zI-D_{G})^{-1}$ are most the corresponding entries of $(D_{G}+(\lambda+\gamma)I)^{-1}$. Therefore, if we express $f_{1}(\epsilon_{1},\ldots,\epsilon_{n})$ as $\sum_{S\subset [n]}a_{S}\prod_{k\in S}\epsilon_{k}$ and express $f_{2}(\epsilon_{1},\ldots,\epsilon_{n})$ as $\sum_{S\subset [n]}b_{S}\prod_{k\in S}\epsilon_{k}$, then we have $|a_{S}|\leq b_{S}$ for all $S\subset [n]$.

By Theorem \ref{thm2}, if
\[g(\epsilon_{1},\ldots,\epsilon_{n})=\frac{1}{2}\sum_{r=1}^{n}(f_{2}(\epsilon_{1},\ldots,\epsilon_{n})-
f_{2}(\epsilon_{1},\ldots,\epsilon_{r-1},-\epsilon_{r},\epsilon_{r+1},\ldots,\epsilon_{n})),\]
then for all $p\in\mathbb{N}$ and all choices of $X_{1}^{(i)},\ldots,X_{n}^{(i)}\in\{-1,1\}$, for $i=1,\ldots,p$, we have
\begin{equation}\label{eq5}
\left|\frac{1}{p^{2}}\sum_{i,j\in[p]}f_{1}(X_{1}^{(i)}X_{1}^{(j)},\ldots,X_{n}^{(i)}X_{n}^{(j)})-\mathbb{E}f_{1}\right|\leq
\frac{1}{p^{2}}\sum_{i,j\in[p]}g(X_{1}^{(i)}X_{1}^{(j)},\ldots,X_{n}^{(i)}X_{n}^{(j)}).
\end{equation}
If the $X_{k}^{(i)}$ are chosen independently and uniformly from $\{-1,1\}$, then
\begin{equation}\label{eq6}
\mathbb{E}\frac{1}{p^{2}}\sum_{i,j\in[p]}g(X_{1}^{(i)}X_{1}^{(j)},\ldots,X_{n}^{(i)}X_{n}^{(j)})=\frac{1}{p}g(1,\ldots,1).
\end{equation}
Since $-\Delta$ is positive semidefinite, we have
\begin{align*}
&|f_{2}(\epsilon_{1},\ldots,\epsilon_{n})-f_{2}(\epsilon_{1},\ldots,\epsilon_{r-1},-\epsilon_{r},\epsilon_{r+1},\ldots,\epsilon_{n})|\\\leq&
\frac{d+\lambda+\gamma}{2\pi}\left(\int_{0}^{2\pi}|h(-d+(d+\lambda+\gamma)e^{it})|\,dt\right)\frac{2\lambda}{n\gamma^{2}},
\end{align*}
for all $\epsilon_{1},\ldots,\epsilon_{n}\in\{-1,1\}$ and $r\in[n]$. Thus,
\[g(1,\ldots,1)\leq\frac{(d+\lambda+\gamma)\lambda}{2\pi\gamma^{2}}\left(\int_{0}^{2\pi}|h(-d+(d+\lambda+\gamma)e^{it})|\,dt\right).\]
By (\ref{eq6}), we have
\[\mathbb{E}\frac{1}{p^{2}}\sum_{i,j\in[p]}g(X_{1}^{(i)}X_{1}^{(j)},\ldots,X_{n}^{(i)}X_{n}^{(j)})\leq
\frac{1}{p}\frac{(d+\lambda+\gamma)\lambda}{2\pi\gamma^{2}}\left(\int_{0}^{2\pi}|h(-d+(d+\lambda+\gamma)e^{it})|\,dt\right).\]
Thus, if $p$ is large enough, then with high probability, one is able to assert a good estimation of $\mathbb{E}f_{1}$ using (\ref{eq5}).
\end{example}
{\bf Acknowledgement:} The author is supported by NSF DMS-1856221.

\end{document}